\documentclass[11pt,a4paper]{article}
\usepackage[utf8]{inputenc}
\usepackage[square,numbers]{natbib}
\usepackage[english]{babel}
\usepackage{amsmath,amsthm}
\usepackage{amsfonts}
\usepackage{amssymb}
\usepackage{graphicx}
\usepackage{float}
\usepackage{overpic}
\usepackage[margin=1.2in]{geometry}
\usepackage{enumerate}
\usepackage{marginnote}
\usepackage[dvipsnames]{xcolor}
\usepackage{soul}
\usepackage[normalem]{ulem}
\usepackage{cleveref}

\theoremstyle{plain}
\newtheorem{theorem}{Theorem}[section]
\newtheorem*{ThmA}{Theorem A}
\newtheorem*{ThmB}{Theorem B}

\newtheorem*{ThmD}{Theorem D}

\newtheorem*{CorC}{Corollary C}

\newtheorem{corollary}[theorem]{Corollary}
\newtheorem{lemma}[theorem]{Lemma}
\newtheorem{proposition}[theorem]{Proposition}
\newtheorem*{conjecture}{Conjecture} 

\theoremstyle{definition}
\newtheorem{definition}[theorem]{Definition}
\newtheorem{remark}[theorem]{Remark}

\newtheorem{question}[theorem]{Question}

\newcommand{\Z}{\ensuremath{{\mathbb{Z}}}}

\newcommand{\G}{\Gamma}
\newcommand{\Ge}{\Omega}

\begin{document}

	\title{Acylindrical hyperbolicity for Artin groups with a visual splitting}
		\author{Ruth Charney, Alexandre Martin, and Rose Morris-Wright}
		\date{}

	\maketitle
	
	\begin{abstract}
	We establish a criterion that implies the acylindrical hyperbolicity of many Artin groups admitting a visual splitting. This gives a variety of new examples of acylindrically hyperbolic Artin groups, including many Artin groups of FC-type.

	Our approach  relies on understanding when parabolic subgroups are weakly malnormal in a given Artin group. We formulate a conjecture for when this happens, and prove it for several classes of Artin groups, including all spherical-type, all two-dimensional, and all even FC-type Artin groups. In addition, we establish some connections between several conjectures about Artin groups, related to questions of acylindrical hyperbolicity, weak malnormality of parabolic subgroups, and intersections of parabolic subgroups. 
	\end{abstract}

\section{Introduction}

\paragraph{Background and motivation.} Artin groups are generalizations of braid groups, with many connections to Coxeter groups.  
Artin groups remain largely mysterious in general, both from an algebraic and geometric viewpoint, although significant progress has been made in studying specific classes: spherical-type, FC-type, and 2-dimensional Artin groups, etc. (see \Cref{Background} for the definition of each class.) For Artin groups outside of  these classes, it remains unknown in general whether they have solvable word problem, contain torsion elements, or have non-trivial centers.  Geometrically, while no Artin groups other than free groups are hyperbolic (as they otherwise contain $\Z^2$-subgroups), it is expected that they are all CAT(0), although this is still open even for braid groups. Some classes of Artin groups have been shown to satisfy other notions of non-positive curvature, see for instance \cite{HaettelXXL,HagenMartinSistoHHG, HuangOsajdaSystolic,HuangOsajdaHelly}

Acylindrical hyperbolicity is a notion encapsulating the idea of a group `having hyperbolic directions', a very weak form of hyperbolic behaviour. Many groups of geometric interest are known to be acylindrically hyperbolic, and despite its generality, this notion  is strong enough to have important consequences for the structure of the group.  (We refer the reader to \cite{Osin2016, Osin2017} for a discussion of the consequences.)

In the case of Artin groups, there is a clear conjectural picture of when they are expected to be acylindrically hyperbolic: 
\begin{conjecture}[Acylindrical Hyperbolicity Conjecture]
	Let $A_\G$ be an irreducible Artin group. Then the central quotient $A_G / Z(A_\G)$ is acylindrically hyperbolic. 
\end{conjecture}

This conjecture essentially states that Artin groups are expected to be acylindrically hyperbolic unless they ``clearly cannot be''. Indeed, reducible Artin groups cannot be acylidrically hyperbolic as they split as direct products of infinite groups. Spherical-type Artin groups cannot be acylindrically hyperbolic as they have an infinite cyclic centre (but the acylindrical hyperbolicity of their central quotient was proved by Calvez and Wiest in \citep{CalvezWiest2017Acylindrical}.) 
Note that Artin groups of non-spherical type are conjectured to have a trivial centre, so in that case the conjecture states that an Artin group  of non-spherical type is  acylindrically hyperbolic if and only if it is irreducible. This conjecture has been proved for several families of Artin groups already, see \Cref{Background} for details.

Beside being interesting in its own right, the question of acylindrical hyperbolicity for Artin groups has applications to some well-known open problems for these groups. A first possible application is to the centre of these groups, as Artin groups of infinite type are conjectured to have a trivial centre. Since it is known that acylindrically hyperbolic groups have a finite centre, showing this property is a possible first step towards proving the triviality of the centre. Another possible application comes from the isomorphism problem, which asks which labelled graphs produce isomorphic Artin groups. Very little is currently known about the Isomorphism Problem for Artin groups, see for instance \cite{Droms1987,Paris2004,Crisp2005, Vaskou2023isomorphism,MartinVaskou2024}. For instance, it is not even known whether being a spherical-type Artin group, or being an irreducible Artin group, is invariant under isomorphism. Since acylindrically hyperbolic groups have finite centres and do not split as direct products of infinite factors, a positive answer to the Acylindrical Hyperbolicity Conjecture would imply that both aforementioned properties are indeed invariant under isomorphism.

A particular family of Artin groups all of whose elements are expected to be acylindrically hyperbolic is the family of Artin groups whose presentation is not a complete graph. Such Artin groups have the useful feature of decomposing as amalgamated products of standard parabolic subgroups (over a standard parabolic subgroup). Such splittings, which we will refer to as \textit{visual splittings} as they can be read directly from the presentation graph, have been used to derive properties of the Artin group from the properties of the corresponding parabolic subgroups, see for instance \cite{CharneyDavis1995,GodelleParis2012Basic,MorrisWright2019,moller_parabolic_2023}. 
For groups splitting as amalgamated products, and more generally for groups acting on trees, there is a useful acylindrical hyperbolicity criterion due to Minsayan--Osin \cite{minasyan_acylindrical_2015}. In a nutshell (see \Cref{thm: M-O} for the precise statement), a non-virtually cyclic group $G$ splitting as an amalgamated product $G = A*_C B$ is acylindrically hyperbolic as soon as the edge group $C$ is \textit{weakly malnormal} in $G$, i.e. as soon as $C$ intersects one of its conjugates along a finite subgroup. In the case of Artin groups admitting a visual splitting, this amounts to understanding when standard parabolic subgroups are weakly malnormal in the ambient group. We introduce the following conjecture, which provides a complete description of when this is expected to happen: 

\begin{conjecture}[Weak Malnormality Conjecture]
	A proper standard parabolic subgroup of $A_\G$ is weakly malnormal if and only if it does not contain a standard parabolic subgroup that is a direct factor of~$A_\G$.
	
	In particular, if $A_\G$ is irreducible, then every proper standard parabolic subgroup is weakly malnormal.
\end{conjecture}

In trying to apply the criterion of Minsayan--Osin, we are thus led to study the intersections of standard parabolic subgroups. Such intersections have been heavily studied in recent years in connection with other problems about Artin groups. In particular, the following conjecture is particularly relevant: 

\begin{conjecture}[Intersection Conjecture]
	The intersection of any two parabolic subgroups in $A_\G$ is again a parabolic subgroup. 
\end{conjecture}

This conjecture has been proved for a few families of Artin groups but remains open in general. For some families of Artin groups, certain weaker versions have been established, see \Cref{Background} for more details. 

\medskip

In this paper, we study the connections between these three conjectures, and show the acylindrical hyperbolicity of new classes of Artin groups.

\paragraph{Statement of results.} We now state the main results of this article. Our main theorem is a criterion for showing the acylindrical hyperbolicity of an Artin group admitting a visual splitting, under a mild assumption on the amalgamating subgroup:

\begin{ThmA}
		Let $A_\Gamma$ be an irreducible Artin group that splits visually as an amalgamated product $A_\Gamma = A_{\Gamma_1} *_{A_{\Ge}} A_{\Gamma_2}$. If the intersection of any two conjugates of $A_{\Ge}$ is again a parabolic subgroup of $A_{\Gamma}$, then $A_{\Ge}$ is weakly malnormal in $A_\Gamma$. In particular, $A_\Gamma$ is acylindrically hyperbolic. 
	\end{ThmA}

As a consequence of this result, we also obtain the following result showing the connection between the three conjectures at the centre of this article: 

\begin{ThmB}
	Suppose $A_\G $ is irreducible and $\G$ is not a clique. 
	\begin{itemize}
		\item If $A_\G$ satisfies the Intersection Conjecture, then it also satisfies the Weak Malnormality Conjecture. 
		\item If $A_\G$ satisfies the Weak Malnormality Conjecture, then it also satisfies the Acylindrical Hyperbolicity Conjecture.  
		\end{itemize}
	\end{ThmB}

Theorem~A can be used to show the acylindrical hyperbolicity of many new classes of Artin groups. For instance, using the existing results about intersections of parabolic subgroups in Artin groups of FC type \cite{antolin_subgroups_2023}, we obtain the following: 

\begin{CorC}
	Even Artin groups of FC type satisfy the Acylindrical Hyperbolicity Conjecture. In addition, any
	Artin group of FC type that visually splits over a spherical-type parabolic subgroup satisfies the Acylindrical Hyperbolicity Conjecture. 
\end{CorC}

Note that in Theorem~A, the condition on the edge group is strictly weaker than requiring that the whole Artin group $A_\G$ satisfies the Intersection Conjecture. In \Cref{Sec:weaker IP}, we give examples of how to check this property in some cases, by using a framework of Godelle--Paris \cite{GodelleParis2012KPi1} to construct suitable CAT(0) cube complexes associated to Artin groups. As an application, we derive the acylindrical hyperbolicity of some new Artin groups whose underlying graph is a cone (\Cref{cor:AH_wheel}), and for which the Intersection Conjecture is currently unknown. 

\medskip 

In some special cases,  the main hypothesis of Theorem~A may be verified simply  by checking that the edge group is weakly malnormal in one of the vertex groups. We thus provide a list of Artin groups for which we know the Weak Malnormality Conjecture, as this allows us to quickly verify this condition, and also allows us to construct new examples of acylindrically hyperbolic Artin groups (see \Cref{cor: AH for spherical even FC and Euclidean}):

\begin{ThmD}
		The Weak Malnormality Conjecture holds for the following classes of groups: 
	\begin{itemize}
		\item Artin groups satisfying the hypothesis of Theorem~A (for instance, even Artin groups of FC type),
		\item Artin groups of spherical type, 
		\item two-dimensional Artin groups.
	\end{itemize}
\end{ThmD}

\paragraph{Organisation of the paper.} In Section~\ref{Background}, we recall the terminology and some standard results about Artin groups and their parabolic subgroups. In Section~\ref{sec:AH}, we prove Theorem~A and Corollary~C by studying in detail the action of these Artin groups on their Bass--Serre trees. We also use CAT(0) cube complexes introduced by Godelle--Paris to prove the acylindrical hyperbolicity of additional classes of Artin groups. In Section~\ref{sec:WM}, we prove Theorems~B and~D by studying the geometry of the orbits of parabolic subgroups in a suitable complex (depending on the case: Bass-Serre tree of a splitting, or Deligne complex of the group).

\paragraph{Acknowledgements.} Part of this work was conducted during the 2023 ICMS workshop ``Polyhedral Products: a Path Between Homotopy Theory and Geometric Group Theory'', as well as during the 2023 AIM workshop ``Geometry and topology of Artin groups''. The authors warmly thank the organisers, as well as the personnel of these institutes, for providing a supportive and mathematically rich environment.

\section{Preliminaries on Artin groups} \label{Background}

\begin{definition}
	Let $\G$ be a graph with vertices labeled by the set $S$ and any edge between $s$ and $t$ labeled by an integer $m_{st}\in \{2,3,\dots\}.$ Define the Artin group by the presentation  $$A_\G= \langle S \mid 
	\underbrace{stst\dots}_{m_{st}\text{ terms}}=\underbrace{tsts\dots}_{m_{st}\text{ terms}}\ 
	\textrm{ for all edges in } \G\rangle$$ 
\end{definition}

Note that in this definition, two vertices $s$ and $t$ which are not joined by an edge in~$\G$ have the free relation and in this case we denote $m_{st}=\infty$. The graph $\G$ is oftern called the presentation graph. In some literature, a different defining graph is used called the Dynkin Diagram, wherein edges with $m_{st}=2$  are omitted, while edges with  $m_{st}=\infty$ are included. 

For every Artin group, there is an associated Coxeter group, which is obtained  by adding to the Artin presentation the relation $s^2=1$ for all generators $s$. An Artin group is called \textit{spherical-type} or \textit{finite-type} if the corresponding Coxeter group is finite. Such Artin groups have well-understood algebraic and geometric properties in comparison to their infinite-type cousins. 

The \textit{dimension} of an Artin group $A_\G$ is the maximal size of a subgraph $\Omega \subset \G$ such that $A_\Omega$ is spherical-type.  So for example, a 2-dimensional Artin group is one for which the only spherical-type parabolic subgroups are either cyclic (1-generator) or dihedral (2-generator) Artin groups.

\begin{definition}
	An Artin group $A_\G$ is said to be: 
	\begin{itemize}
		\item of \textit{spherical type} if the corresponding Coxeter group $W_\G$ is finite,
		\item of \textit{FC type} if for every induced complete subgraph $\Gamma' \subset \Gamma$, the corresponding Coxeter group $W_{\Gamma'}$ is finite,
		\item \textit{two-dimensional} if for every triangle of $\G$ with vertices $a, b, c$, we have 
		$$\frac{1}{m_{ab}} + \frac{1}{m_{bc}} + \frac{1}{m_{ac}} \leq 1$$
		(This is equivalent to requiring that $A_\G$ has cohomological dimension $2$.)
		\item of \textit{even type} if all labels of $\G$ are even.
	\end{itemize}
\end{definition}

\subsection{Parabolic subgroups}

For an induced subgraph $\Ge$ of $\G$, the Artin group $A_{\Ge}$ embeds as a 
%convex 
subgroup of $A_\G$ \cite{VanderLek1983}.
We call such a subgroup a \textit{standard parabolic subgroup}.  A \textit{parabolic subgroup} is a conjugate of some standard parabolic subgroup.  
Parabolic subgroups play a central role in our understanding of Artin groups.  For example, many conjectures about Artin groups can be reduced to the case of parabolic subgroups corresponding to cliques in $\G$, that is, if $A_{\Ge}$ satisfies the conjecture  for every clique $\Ge$ in $\G$, then the conjecture also holds for $A_\G$ (\cite{GodelleParis2012Basic, MorrisWright2019}). 

We say that an Artin group $A_\G$ is \textit{reducible} if there exist two induced disjoint subgraphs $\G_1, \G_2$ of $\G$ such that $V(\G) = V(\G_1) \cup V(\G_2)$, and every vertex of $\G_1$ is connected to every vertex of $\G_2$ by an edge labelled $2$. In that case, we have that $A_\G$ decomposes as the direct product $A_{\G_1} \times A_{\G_2}$, and the parabolic subgroups $A_{\G_1}, A_{\G_2}$ are called \textit{direct factors} of $A_\G$. If no such subgraphs $\G_1, \G_2$ exist, the Artin group~$A_\G$ is called \textit{irreducible}. 

The behavior of parabolic subgroups will be key to the discussion which follows.  As noted above, we do not know in general if intersections of parabolic subgroups are parabolic.  A useful fact proven by Blufstein and Paris \cite{blufstein_parabolic_2023}, is that if  $P \subseteq P'$ are two parabolic subgroups of $A_\G$, then $P$ is also a parabolic subgroup of $P'$.  
The following lemma will be often used in this article.

\begin{lemma}\label{lem: bounded chain parabolics}
	Let $A_\Gamma$ be an Artin group. Then for every sequence of parabolic subgroups $H_0 \subsetneq \cdots \subsetneq H_n$, we have $n \leq |V(\Gamma)|$. In particular, there is an upper bound on the length of chains of parabolic subgroups. 
\end{lemma}

\begin{proof}  Say $H_i$ is a conjugate of $A_{\Omega_i}$.  By Blufstein--Paris \cite{blufstein_parabolic_2023}, for each $i$, $H_{i-1}$ is a (proper) parabolic subgroup of $H_i$, so $\Omega_{i-1}$ must be a proper subgraph of $\Omega_i$, that is, $\Omega_0 \subsetneq \cdots \subsetneq \Omega_n \subseteq \G$.  The result now follows.
\end{proof}

A geometric construction that has played a primary role in the study of Artin groups is the Deligne complex.  For an infinite-type Artin group $A_\G$, let $\mathcal P_\G$ denote the poset consisting of cosets $aA_T \subset A_\G$ such that $A_T$ is a spherical-type parabolic subgroup.  Partially order  $\mathcal P_\G$  by inclusion.  The \textit{Deligne complex}, $D_\G$ is the cell complex whose vertices are the elements of $\mathcal P_\G$ and whose cells are cubes spanned by intervals $[aA_T, aA_{T'}]$  for pairs $aA_T \subseteq aA_{T'}$. The Artin group acts on the Deligne complex by left multiplication, and the vertex corresponding to the coset $aA_T$ is stabilized by the parabolic subgroup~$(A_T)^a := aA_Ta^{-1}$.

There are two well-known metrics on $D_\G$.  One is the standard cubical metric; this metric is CAT(0) if and only if $A_\G$ is FC-type. The other is a piecewise Euclidean metric, called the Moussong metric, in which the metric on a cube $[aA_T, aA_{T'}]$ depends on the shape of the Coxeter cell for $W_{T'}$.  It is conjectured that the Moussong metric is CAT(0) for all infinite-type Artin groups.  This has been shown to hold for all 2-dimensional Artin groups \cite{CharneyDavis1995},  some 3-dimensional Artin groups \cite{Charney2004}, and a class known as locally reducible Artin groups \cite{CharneyLocallyReducible}.

\subsection{Visual splittings}

Acylindrically hyperbolic groups do not have infinite direct factors. Likewise a group that factors as a direct product clearly has  proper subgroups which are not weakly malnormal. Thus, we will focus on irreducible Artin groups in this paper. 
We will also focus on Artin groups that can be decomposed as amalgamated products. 

\begin{definition}[Visual splitting of an Artin group] A \emph{visual splitting} of an Artin group $A_\G$ is a splitting as an amalgamated free product $A_\G = A_{\Gamma_1} *_{A_{\Ge}} A_{\Gamma_2}$ where $\G_1$ and $\G_2$ are proper, full subgraphs of $\G$ and $\Ge = \G_1 \cap \G_2$. This happens precisely when $\Gamma = \Gamma_1 \cup \Gamma_2$ and every edge between a vertex of $\Gamma_1$ and a vertex of $\Gamma_2$ is contained in $\Ge$. We will always assume that the visual splittings we consider are \textit{non-trivial}, i.e. $\Ge \neq \G_1, \G_2$.
\end{definition}

	Note that such a splitting exists if and only if $\G$ is not a clique. Indeed, assume that the two vertices $s, t \in V(\G)$ are not connected by an edge in $\G$. Then one checks that $A_\G$ splits as the amalgamated product 
	$A_\G = A_{\Gamma_1} *_{A_{\Ge}} A_{\Gamma_2}$ with $\G_1 = \G - \{s\}$, $\G_2 = \G - \{t\}$, and $\Ge = \G - \{s, t\}$. 

When an Artin group $A_\G = A_{\Gamma_1} *_{A_{\Ge}} A_{\Gamma_2}$ admits a visual splitting, we can sometimes derive properties of $A_\G$ using properties of $A_{\Gamma_1}, A_{\Gamma_2}$, and $A_{\Ge}$. 
This can for instance be done to study Artin groups of FC type, as such groups can always be decomposed as a sequence of nested amalgamated free products, where the final splitting has spherical-type edge groups. 

\subsection{Existing results for the main conjectures} \label{sec:previous results}

In this section we will review known results for the main conjectures, and prove some elementary implications that we will use later. 

\paragraph{Intersection Conjecture.}
 In 1983, Van der Lek showed that the intersection of standard parabolic subgroups $A_{\G_1}\cap A_{\G_2}$ is always parabolic \citep{VanderLek1983}. More recently, the more general Intersection Conjecture, namely the property that the intersection of any two parabolics is again parabolic, has been proven to hold for the following classes of Artin groups:

\begin{itemize}
	\item Spherical-type Artin groups \cite{CumplidoGebhardtGonzalesMenesesWiest2017}
	\item Right-angled Artin groups and other graphs of groups \cite{Antolin2015}
	\item Large type (i.e. all labels satisfy $m_{s,t}\geq 3$) Artin groups \cite{CumplidoMartinVaskou2020}
	\item (2,2)-free two-dimensional Artin group, i.e. $\G$ does not have two consecutive edges labeled by 2 and the cohomological dimension of $A_\G$ is 2. \cite{Blufstein2022}

\item Euclidean type of the form $\tilde{A}_n$ and $\tilde{C}_n$ \cite{haettel_lattices_2023}
\item Even FC type Artin groups \cite{antolin_subgroups_2023}
	\end{itemize}
	
While this conjecture remains open in general, there are several other classes of Artin groups for which some weaker version of the intersection property is known to hold.  For example, in \cite{MorrisWright2021} it is shown that in FC-type Artin groups, intersections of two spherical-type parabolics are parabolic, and this was further generalized by M\"oller, Paris, and Varghese in  \cite{moller_parabolic_2023} to include the case where just one of the two parabolics is spherical-type. 
	
	The proof  in \cite{MorrisWright2021} that intersections of spherical-type parabolics are parabolic for FC-type Artin groups uses the fact that the cubical metric on the Deligne complex $D_\G$ is CAT(0) for these groups.  This argument can be generalized to other Artin groups acting on CAT(0) spaces where the intersection property is known for the stabilizers of  vertices. We include the proof here for completeness.

	\begin{proposition}\label{lem: IP via Deligne} Let $A_\G$ be an Artin group acting on a polyhedral complex $X$  with a piecewise Euclidean CAT(0) metric, where each cell stabiliser is a parabolic subgroup of~$A_\Gamma$. Assume that the action is without inversions, that is, the stabilizer of each cell pointwise fixes the entire cell.  
		Let $\mathcal{P}$ be the collection of parabolic subgroups that appear as stabilizers of the cells of $X$. Suppose that the stabiliser of every vertex of $X$ satisfies 
		the Intersection Conjecture. Then the intersection of any two elements of $\mathcal{P}$ is again a parabolic subgroup of $A_\Gamma$.
	\end{proposition}

	\begin{proof}
	 Let $P$ and $P'$ be two parabolics in $\mathcal{P}$.  There exist cells $\sigma$ and $\sigma'$ in $X$ such that $P,P'$ are the stabilizers of $\sigma, \sigma'$ respectively.  
	 Let $x, x'$ be two points in the interior of $\sigma, \sigma'$ respectively. Since $P \cap P'$ fixes both $\sigma$ and $\sigma'$, it fixes $x$ and $x'$, and hence the unique CAT(0) geodesic $\gamma$ between them.  Moreover, because the action is without inversions, it fixes the subcomplex consisting of the union of all the cells that contain a point of $\gamma$ in their interior.  In particular, it fixes some edge path $\rho$ that contains all vertices of $\sigma$ and $\sigma'$. Thus, $P \cap P'$ is equal to the pointwise stabiliser of $\rho$. 
	 The result now follows from the following: 
	 
	 \medskip
	 
	 \textbf{Claim.} Let $v_0, \ldots, v_k$ be a combinatorial path in $X$, with stabilisers $P_0, \ldots, P_k$ respectively. Then the intersection $\bigcap_{0 \leq i \leq k} P_i$ is a parabolic subgroup of $A_\G$. 
	 
	 \medskip

	 Let us prove this claim by induction on $k\geq 1$. For $k=1$, the  intersection $P_0 \cap P_1$ is the stabilizer of a single edge. By hypothesis, the stabilizer of the edge is a parabolic subgroup in $\mathcal{P}$.

	 Now suppose that we have proved the result for some $k\geq 1$. Consider a combinatorial path $v_0, \ldots, v_{k+1}$ with stabilisers $P_0, \ldots, P_{k+1}$ . By the induction hypothesis, we have that $\cap_{0 \leq i \leq k} P_i$ is a parabolic subgroup of $A_\G$. Note that this is a subgroup of $P_k$, and hence a parabolic subgroup of $P_k$ by \cite{blufstein_parabolic_2023}. We also know that $P_k \cap P_{k+1}$ is  the stabilizer of the  last edge of this path, and so $P_k\cap P_{k+1}$ is a parabolic subgroup of $P_k$. 
	 
	 We can thus write
	 
	 \[\bigcap_{0 \leq i \leq k+1} P_i=(\bigcap_{0 \leq i \leq k} P_i)\cap (P_k\cap P_{k+1})\]

	Since the Intersection Conjecture holds in $P_k$ by assumption, it follows that  $\bigcap_{0 \leq i \leq k+1} P_i$  is a parabolic subgroup of $P_k$, and hence of $A_\G$. 
	\end{proof}

	In particular, for Artin groups for which the Moussong metric on Deligne complex is known to be CAT(0), the lemma above applies to show that intersections of spherical-type parabolics are parabolic. As noted above, this holds for 2-dimensional Artin groups, some 3-dimensional Artin groups, and locally reducible Artin groups.
	
	\begin{remark}
Note that the above proposition can be generalised to actions on complexes that satisfy other forms of nonpositive curvature, as in \cite{CumplidoMartinVaskou2020, Blufstein2022}. We leave it to the reader to check that the key geometric feature necessary  for the  proof to carry over is the following property: If an element $g \in A_\G$ fixes two vertices $v, v'$ of $X$, then it also fixes some combinatorial path of $X$ from $v$ to $v'$. Such a weak form of convexity is satisfied by many forms of non-positive curvature.
\end{remark}

	We can also obtain more examples of groups that satisfy the Intersection Conjecture by taking products of groups where the Intersection Conjecture is known. 
	
	\begin{lemma}\label{lem: IP reducible}
		Suppose that $A_\G$ is a reducible Artin group with direct factors $A_\G=A_{\G_1}\times \dots \times A_{\G_k}$. If the Intersection Conjecture holds for all parabolic subgroups in each direct factor $A_{\G_i}$ then the Intersection Conjecture holds for $A_\G$. 
	\end{lemma}
	
	\begin{proof}
	Given two parabolics $P, Q$ of $A_\G$, one can decompose them as direct products
	$$ P = P_1  \times \cdots \times P_k, ~~~Q = Q_1  \times \cdots \times Q_k$$
	with each $P_i, Q_i$ a parabolic subgroup of $A_{\G_i}$. We thus have 
	$$ P \cap Q = (P_1 \cap Q_1) \times \cdots \times (P_k \cap Q_k)$$
	and the result  follows from the fact that each $A_{\G_i}$ satisfies the Intersection Conjecture. 
	\end{proof}

We also add the following examples, which will be used later in this article. 
	
	\begin{lemma} \label{lem:IP for 3gen}
		Suppose that $A_\G$ is an Artin group with 3 or fewer generators. Then $A_\G$ satisfies the Intersection Conjecture. 
	\end{lemma}
	
	\begin{proof}
		The result is clear if $A_\G$ has one generator. If it has two generators, it is either a free group (in particular, a right-angled Artin group) or a dihedral Artin group (in particular, a spherical-type Artin group), and so the result follows from \cite{Antolin2015} and \cite{CumplidoGebhardtGonzalesMenesesWiest2017} respectively. 
		
	Let us assume that $A_\G$ has three generators.  If $A_\G$ is spherical-type, this holds by \cite{CumplidoGebhardtGonzalesMenesesWiest2017}, so assume it is infinite type.  If $\G$ is complete, then the Moussong metric on the Deligne complex $D_\G$ is CAT(0) and all proper parabolics are spherical-type, hence appear as stabilisers of cells in $D_\G$.  If $\G$ is not complete, we can instead use the augmented Deligne complex introduced in \cite{Vaskou2022},  which also admits a CAT(0) metric and has all proper parabolics as stabilisers of cells. The result now follows from Proposition~\ref{lem: IP via Deligne}.
	\end{proof}

\paragraph{Acylindrical Hyperbolicity.}
In this paragraph, we review previously known results about acylindrically hyperbolic Artin groups. 
For Artin groups of spherical type, the Acylindrical Hyperbolicity Conjecture was proved by Calvez--Wiest \cite{CalvezWiest2017Acylindrical}, following earlier result for braid groups \cite{MasurMinsky1999,Bowditch2008}. Thus, the Acylindrical Hyperbolicity Conjecture reduces to the case of Artin groups of infinite  type. Such groups are conjectured to have a trivial centre, so the conjecture asks whether these groups are acylindrically hyperbolic.
Currently, the Acylindrical Hyperbolicity Conjecture is known for several classes of Artin groups (we refer the reader to these articles for the definition of some of these classes), which we list below:

\begin{itemize}
	\item Spherical-type Artin groups, by Calvez--Wiest \citep{CalvezWiest2017Acylindrical}
	\item Right-angled Artin groups, by Osin \cite{Osin2016}.
	\item Two-dimensional Artin groups, by Vaskou \cite{Vaskou2022}, following earlier work for XXL-type Artin groups (i.e. all labels satisfy $m_{s,t}\geq 5$) by Haettel \cite{HaettelXXL},  for XL-type Artin groups (i.e. all labels satisfy $m_{s,t}\geq 4$) by Martin--Przytycki \cite{martin_acylindrical_2021}, and for some two-dimensional Artin groups admitting a specific CAT(0) model by Kato--Oguni \cite{KatoOguni2022}.
	\item Euclidean-type Artin groups, by Calvez \cite{Calvez2022}.
	\item Artin groups whose graph is not a join, by Charney--Morris-Wright \cite{MorrisWright2019}, following previous work by Chatterji--Martin \cite{ChatterjiMartin2019}.
	\item Some relatively extra-large type Artin groups, by Goldman \cite{Goldman2022}.
	\item Some locally reducible Artin groups, by Mastrocola \cite{JillThesis}.
\end{itemize}

In this article, we add to this list the class of even Artin groups of FC type, among other new examples.

\section{Artin groups with visual splittings and acylindrical hyperbolicity} \label{sec:AH} \label{Sec:Implications of Intersection conjecture}

The goal of this section is to obtain new criteria for proving acylindrical hyperbolicity and apply them to get new examples of acylindrically hyperbolic Artin groups.
We will focus primarily on the case of Artin groups with a visual splitting.  In this case, there is a clear connection between malnormality and acylindricity given by the following theorem of Minasyan and Osin \cite{minasyan_acylindrical_2015}.

\begin{theorem}\label{thm: M-O}[see \cite{minasyan_acylindrical_2015}]  Suppose $G$ splits as an amalgamated product of groups $G=A *_C B$ with $A \neq C \neq B$.  If $C$ is weakly malnormal in $G$, then $G$ is either virtually cyclic or acylindrically hyperbolic.
\end{theorem}

Thus a key to proving acylindricity for $A_\G$ is understanding when parabolic subgroups are weakly malnormal.

\subsection{The main acylindrical hyperbolicity criterion}

The main result of this subsection is the following.

\begin{theorem}\label{thm: ah splitting IP edge}
	Let $A_\Gamma$ be an irreducible Artin group that splits visually as an amalgamated product $A_\Gamma = A_{\Gamma_1} *_{A_{\Ge}} A_{\Gamma_2}$. If the intersection of any two conjugates of $A_{\Ge}$ is again a parabolic subgroup of $A_{\Gamma}$, then $A_{\Ge}$ is weakly malnormal in $A_\Gamma$. In particular, $A_\Gamma$ is acylindrically hyperbolic. 
\end{theorem}

In order to prove this result, we need the following characterisation of normal parabolic subgroups of Artin groups. We start with the irreducible case:

\begin{lemma} \label{Lem: no proper normal subgroups}
	Let $A_\Gamma$ be an irreducible Artin group. Then the only normal standard parabolic subgroups of $A_\Gamma$ are $A_\Gamma$ and the trivial subgroup. 
\end{lemma}

\begin{proof}
	Let $ \varnothing \subsetneq \Gamma' \subsetneq \Gamma$ be a strict subgraph of $\Gamma$. Let $\Gamma_D, \Gamma_D'$ be the Dynkin diagrams  corresponding to the presentation graph $\Gamma, \Gamma'$ respectively (i.e. no edge if $m_{s,t}=2$, and edges with label $\infty$ for pairs $s, t$ that are not connected in $\Gamma$). Since $A_\Gamma$ is irreducible, $\Gamma_D$ is connected and $\Gamma_D'$ is a strict induced subgraph of  $\Gamma_D$. Thus, there exists an edge of $\Gamma_D$ connecting a vertex $s\in \Gamma_D'$ and $t \notin \Gamma_D'$. It follows from van der Lek that 
	$$A_{\Gamma'} \cap A_{\{s, t\}} = \langle s\rangle$$
	We can also compute that 
	$$tA_{\Gamma'}t^{-1} \cap A_{\{s, t\}}=t\big(A_{\Gamma'}\cap A_{\{s, t\}}\big)t^{-1}=t\langle s\rangle t^{-1}$$
	
	Thus in order to show that $tA_{\Gamma'}t^{-1} \neq A_{\Gamma'}$ it is sufficient to show that $t$ does not normalize $\langle s \rangle$. Suppose to the contrary that $t\langle s\rangle t^{-1}=\langle s\rangle$. Then $tst^{-1}\in \langle s\rangle$ and $ts=s^kt$ for some $k$. However the Artin monoid embeds in the Artin group, and the length of a monoid element is well-defined. Therefore, if $ts$ and $s^kt$ are two equal monoid elements then $k=1$ and $ts=st$. This contradicts our original assumption that $s,t$ are joined by an edge in the Dynkin Diagram. Thus we know that $t$ does not normalize $\langle s\rangle$ and $A_{\Gamma'}$ cannot be a normal subgroup.
\end{proof}

\begin{corollary} \label{cor: normal iff product}
	Let $A_\Gamma$ be an Artin group. Then a standard parabolic subgroup is normal if and only if it is a product of direct factors of $A_\Gamma$. 
\end{corollary}

\begin{proof}
	Decompose $A_\Gamma$ as a product of irreducible Artin groups $A_\Gamma=A_{\Gamma_1} \times \dots \times A_{\Gamma_k}$. 
	If $A_{\Gamma'}$ is a product of groups of the form $A_{\Gamma'}=A_{\Gamma_{i_1}}\times \dots \times A_{\Gamma_{i_k}}$ then $A_{\Gamma'}$ is a direct factor of $A_\Gamma$ and so must be a normal subgroup. 
	
	Conversely, if $A_{\Gamma'}$ is normal in $A_\Gamma$ then for all $i$, $A_{\Gamma'}\cap A_{\Gamma_i}$ is normal in $A_{\Gamma_i}$. By Lemma \ref{Lem: no proper normal subgroups}, we see that $A_{\Gamma'}\cap A_{\Gamma_i}$ must be equal to either $A_{\Gamma_i}$ or to the trivial group. This implies that $A_{\Gamma'}$ is a direct product of factors of $A_\Gamma$. 
\end{proof}

\begin{proof}[Proof of Theorem~\ref{thm: ah splitting IP edge}]
	Let $A_\Gamma = A_{\Gamma_1} *_{A_{\Ge}} A_{\Gamma_2}$ be a visual splitting and let $T$ be the Bass--Serre tree of this splitting. 
	
	\medskip 
	
	\textbf{Claim 1:} Let $e$ be an edge of $T$ with vertices $v$ and $w$. The trivial subgroup is the only parabolic subgroup of $\mathrm{Stab}(e)$ that is normal in both $\mathrm{Stab}(v)$ and~$\mathrm{Stab}(w)$.
	
	\medskip
	
Up to conjugation, it is enough to show that the trivial subgroup is the only parabolic subgroup of $A_{\Ge}$ that is normal in both $A_{\Gamma_1}$ and $A_{\Gamma_2}$. Suppose to the contrary that $H$ is normal in both $A_{\Gamma_1}$ and $A_{\Gamma_2}$. Since every element of $A_\G$ is a product of elements of $A_{\Gamma_1}$ and $A_{\Gamma_2}$, $H$ is also normal in $A_\G$. By assumption, $A_\G$ is irreducible, so applying Corollary~\ref{cor: normal iff product} we conclude that $H$ must be trivial. This proves the claim.

	\medskip

	\textbf{Claim 2}: Let $\gamma$ be a geodesic segment of $T$, and let $\mathrm{Stab}_*(\gamma)$ be the pointwise stabiliser of $\gamma$. We claim that if $\mathrm{Stab}_*(\gamma)$ is non-trivial, then we can extend $\gamma$ to  a geodesic segment $\gamma' \supsetneq \gamma$  such that $\mathrm{Stab}_*(\gamma') \subsetneq \mathrm{Stab}_*(\gamma)$.
	
	\medskip

	Let $v_0, \ldots, v_n$ be the vertices of $\gamma$, and for every $1 \leq i \leq n$, let $e_i$ be the edge between $v_{i-1}$ and $v_i$.  The subgroup $H:= \mathrm{Stab}_*(\gamma) = \bigcap_{1 \leq i \leq n}\mathrm{Stab}(e_i)$ is a subgroup of $\mathrm{Stab}_*(e_n)$. Suppose that it is not trivial. There are two cases to consider: 
	\begin{itemize}
		\item If $H$ is not normal in $\mathrm{Stab}(v_n)$, then we pick an element $g \in \mathrm{Stab}(v_n)$ such that $H^g \neq H$. We then set
		$$\gamma' := \gamma \cup g\gamma$$
		and it follows that $\mathrm{Stab}_*(\gamma') = H \cap H^g \subsetneq H = \mathrm{Stab}_*(\gamma)$.
		\item If $H$ is normal in $\mathrm{Stab}(v_n)$, then we pick an element $g \in \mathrm{Stab}(v_n) - \mathrm{Stab}(e_n)$, and define $v_{n+1} := gv_{n-1}$, $e_{n+1} := ge_n$ and 
		$$\gamma' := \gamma \cup ge_n$$
		The pointwise stabilizer of $\gamma'$ is also equal to $H$ since
		$$H = H \cap H^g = \mathrm{Stab}_*(\gamma) \cap \mathrm{Stab}_*(g\gamma) \subseteq \mathrm{Stab}_*(\gamma') \subseteq \mathrm{Stab}_*(\gamma)=H.$$ Note that $\mathrm{Stab}_*(\gamma')$ is a parabolic subgroup of $A_\G$ by assumption on the splitting. Since it is contained in $\mathrm{Stab}_*(e_{n+1})$, it is a parabolic subgroup of $\mathrm{Stab}_*(e_{n+1})$ by Blufstein--Paris \cite{blufstein_parabolic_2023}. It follows from Claim 1 that $H = \mathrm{Stab}_*(\gamma')$ is not normal in~$\mathrm{Stab}(v_{n+1})$, so we can pick an element $g' \in \mathrm{Stab}(v_{n+1})$ such that $H^{g'} \neq H$. We then construct 
		$$\gamma'' := \gamma' \cup g'\gamma'$$
		and it follows that $\mathrm{Stab}_*(\gamma'') = H \cap H^{g'} \subsetneq H = \mathrm{Stab}_*(\gamma)$. This proves the claim.
	\end{itemize}
	
	\medskip

	By Claim 2, we can construct a sequence of geodesic segments $\gamma_0 \subsetneq \gamma_1 \subsetneq \cdots $, such that the sequence of stabilisers $\mathrm{Stab}_*(\gamma_i) $ strictly decreases as long as they are not trivial. By assumption on the splitting, we have that each $\mathrm{Stab}_*(\gamma_i)$ is a parabolic subgroup of $A_\G$. It now follows from Lemma~\ref{lem: bounded chain parabolics} that $\mathrm{Stab}_*(\gamma_i)$ becomes trivial after finitely many steps.

	We thus have a finite length geodesic segment $\gamma$  with trivial point stabilizer.  Let $e_1, \ldots, e_n$ be the edges of $\gamma$. Since $T$ is a tree, we have $\mathrm{Stab}_*(\gamma) = \mathrm{Stab}(e_1) \cap \mathrm{Stab}(e_n)$, and by assumption  $\mathrm{Stab}_*(\gamma) = \{1\}$. Since edge stabilisers are conjugates of $A_{\Ge}$ by construction, it follows that $A_{\Ge}$ is weakly malnormal in $A_{\Gamma}$,  hence $A_\Gamma$ is acylindrically hyperbolic by Theorem~\ref{thm: M-O}. 
\end{proof}

\subsection{Applications to some classes of Artin groups.} 

Recall that $A_\G$ is even FC-type if all edge labels in $\G$ are even, and all cliques in $\G$ generate spherical-type parabolics.

\begin{theorem}\label{thm: even FC-type}  Even FC-type Artin groups satisfy the Acylindrical Hyperbolicity Conjecture. 
\end{theorem}

\begin{proof} Assume that $A_\Gamma$ is irreducible. If $A_\Gamma$ is spherical-type, this follows from Calvez--Wiest \cite{CalvezWiest2017Acylindrical}. 
 If $A_\Gamma$ is not spherical-type, then it admits a visual splitting. Since even FC-type Artin groups satisfy the Intersection Property by \cite{antolin_subgroups_2023}, the intersection of any two conjugates of the edge group for this splitting must be parabolic and so by \Cref{thm: ah splitting IP edge}, $A_\G$ is acylindrically hyperbolic.
\end{proof}

For general FC-type Artin groups, we cannot directly apply Theorem~\ref{thm: ah splitting IP edge} as we do not know that they satisfy the Intersection Conjecture. We therefore ask the following: 

\begin{question}
	Do FC-type Artin groups satisfy the Intersection Conjecture? 
\end{question}

We can nonetheless prove acylindrical hyperbolicity under more restrictive conditions.

\begin{theorem}\label{thm: AH splitting over sph via Deligne}
	Let $A_\Gamma$ be an irreducible Artin group such that the Deligne complex $D_\G$ admits a piecewise Euclidean CAT(0) metric. If $A_\G$ splits visually over a spherical-type parabolic, then $A_\Gamma$ is acylindrically hyperbolic.
\end{theorem}

\begin{proof}
	Stabilisers of simplices in the Deligne complex are precisely the spherical-type parabolic subgroups. Since spherical-type Artin groups satisfy the Intersection Conjecture ~\cite{CumplidoGebhardtGonzalesMenesesWiest2017}, it follows from Proposition~\ref{lem: IP via Deligne} that the intersection of any two spherical-type parabolic subgroups of $A_\Gamma$ is again a parabolic subgroup. In particular, the splitting satisfies the hypothesis of Theorem~\ref{thm: ah splitting IP edge}, and the result follows.
\end{proof}

In particular,  Theorem~\ref{thm: AH splitting over sph via Deligne} applies to FC-type Artin groups, locally reducible Artin groups, and certain Artin groups of dimension $3$ (namely those for which all the irreducible three-dimensional parabolic subgroups are isomorphic to the braid group $B_4$), whenever they split visually over a spherical-type parabolic.

\subsection{A weaker version of the Intersection Property} \label{Sec:weaker IP} In Theorem~\ref{thm: ah splitting IP edge}, the condition on the intersections of conjugates of the edge group $A_{\Ge}$ is \textit{a priori} weaker than requiring that $A_\Gamma$ satisfies the Intersection Conjecture. In this section, we show how to obtain this weaker condition in cases where the Intersection Conjecture may not be known for $A_\Gamma$. 
 We start with the following observation.

\begin{theorem}\label{thm: AH with IP for both vertex groups}
	Let $A_\Gamma$ be an irreducible Artin group with a visual splitting $A_\Gamma = A_{\Gamma_1} *_{A_{\Ge}} A_{\Gamma_2}$. If both $A_{\Gamma_1}$ and $A_{\Gamma_2}$ satisfy the Intersection Conjecture, then the intersection of any two conjugates of $A_\Omega$ is again a parabolic subgroup of $A_\Gamma$. 
	In particular, $A_\Gamma$ is acylindrically hyperbolic. 
\end{theorem}

\begin{proof}
	Since $ A_{\Gamma_1}$ and $ A_{\Gamma_2}$ both satisfy the Intersection Conjecture, the fact that the intersection of any two conjugates of $A_\Omega$ is again a parabolic subgroup of $A_\Gamma$ is a direct consequence of Proposition~\ref{lem: IP via Deligne} applied to the action of $A_\G$ on the (CAT(0)) Bass--Serre tree of the splitting $ A_{\Gamma_1} *_{A_{\Ge}} A_{\Gamma_2}$.
\end{proof}

In the rest of this section, we show how one may understand the intersection of two conjugates of $A_\Ge$ even when the vertex groups are not known to satisfy the Intersection Conjecture. As with the original proofs of the Intersection Property for some families of Artin groups (see \Cref{sec:previous results}), the idea is to realise the conjugates of $A_\Ge$ as stabilisers in a suitable CAT(0) complex and apply Proposition~\ref{lem: IP via Deligne}. We recall a relevant framework of Godelle--Paris to construct such complexes \cite{GodelleParis2012KPi1}.

\begin{definition}
	A \textbf{complete cover} $\mathcal{U}$ of $\Gamma$ is a collection of induced subgraphs of $\Gamma$ that contains every edge of $\Gamma$ and is stable under taking induced subgraphs (including the empty graph). 
	Given an element $\Gamma' \in\mathcal{U}$, the corresponding standard parabolic subgroup $A_{\Gamma'}$ is called a \textbf{$\mathcal{U}$-standard parabolic subgroup}, and a conjugate of $A_{\Gamma'}$ is called a~\textbf{$\mathcal{U}$-parabolic subgroup}.

	Given a complete cover $\mathcal{U}$, one defines the corresponding \textbf{Godelle--Paris cube complex} $X_{\mathcal{U}}$ as follows: Vertices of $X_{\mathcal{U}}$ correspond to cosets of $\mathcal{U}$-standard parabolic subgroups, and cubes correspond to the intervals (for the inclusion) between $gA_{\Gamma_1}$ and $gA_{\Gamma_2}$, whenever $g \in A_\Gamma$ and   $\Gamma_1 \subset \Gamma_2$ are in $\mathcal{U}$. 
	
	Note that $A_\Gamma$ acts on $X_{\mathcal{U}}$ by left multiplication on left cosets. This action is cocompact and without inversion.
\end{definition}

For instance, if $A_\Gamma$ is FC-type and $\mathcal{U}$ consists of all cliques of $\Gamma$, we recover the cubical Deligne complex of Charney--Davis \cite{CharneyDavis1995}. More generally,  
if $\mathcal{U}$ consists of all the cliques of $\Gamma$, one recovers the Godelle--Paris clique complex \cite{GodelleParis2012KPi1}. 

There is a complete characterisation of when  the standard cubical metric on this complex is CAT(0). We need the following definition: 

\begin{definition}
	Let $\mathcal{U}$ be a complete cover of $\Gamma$. We define a simplicial complex $L_{\mathcal{U}}$ as follows: The vertices of $L_{\mathcal{U}}$ are the vertices of $\Gamma$, and a set of vertices $v_0, \ldots, v_k$ of $L_{\mathcal{U}}$ span a $k$-simplex of $L_{\mathcal{U}}$ if and only if the induced subgraph of $\Gamma$ spanned by $v_0, \ldots, v_k$ belongs to $\mathcal{U}$. 
\end{definition}

Note that $L_{\mathcal{U}}$ is isomorphic to the link of any vertex of $X_{\mathcal{U}}$ corresponding to the empty subgraph of $\Gamma$. 

\begin{theorem}[\cite{GodelleParis2012KPi1} Thm 4.2] \label{thm: Godelle Paris flag}
	Let $\mathcal{U}$ be a complete cover of $\Gamma$. 
	Then $X_{\mathcal{U}}$ is a CAT(0) cube complex if and only if $L_{\mathcal{U}}$ is a flag simplicial complex. 
\end{theorem}

\begin{corollary}\label{cor: IP with Godelle Paris}
	Let $A_\Gamma$ be an Artin group. Let $\mathcal{U}$ be a complete cover of $\Gamma$. Assume that: 
	\begin{itemize}
		\item $L_{\mathcal{U}}$ is a flag simplicial complex,
		\item for every  $\Gamma'$ in  $\mathcal{U}$, the standard parabolic $A_{\Gamma'}$ satisfies the Intersection Conjecture. 
	\end{itemize}
Then the intersection of any two $\mathcal{U}$-parabolic subgroups is again a parabolic subgroup of~$A_\Gamma$. 
\end{corollary}

\begin{proof}
	By Theorem~\ref{thm: Godelle Paris flag}, we have that $X_{\mathcal{U}}$ is a CAT(0) cube complex. By construction, the action is without inversion and the stabilisers of cubes are precisely the $\mathcal{U}$-parabolic subgroups of $A_\Gamma$. Thus, the result follows from Proposition~\ref{lem: IP via Deligne}.
\end{proof}

In particular, to show that the intersection of two conjugates of a standard parabolic $A_{\Ge}$ is again a parabolic subgroup, it is enough to include $\Ge$ in a suitable complete cover of $\Gamma$. We now give an example of a geometric condition on $\Ge$ that guarantees  that such a cover exists.

\begin{definition}
	We say that an induced subgraph $\Ge$ is \textbf{$2$-convex} in $\Gamma$ if every geodesic path of $\Gamma$ of length $2$ with endpoints in $\Ge$ is contained in $\Ge$. 
\end{definition}

\begin{lemma}
	Let $\Gamma$ be a simplicial graph, and let $\Ge$ be a $2$-convex subgraph of $\Gamma$. Let $\mathcal{U}$ be the complete cover consisting of all the cliques of $\Gamma$ and all the induced subgraphs of $\Ge$.  Then the simplicial complex $L_\mathcal{U}$ is flag.
\end{lemma}

\begin{proof}
	Let $T \subset V(\Gamma)$ be a set of vertices that are pairwise connected by edges in $L_{\mathcal{U}}$. To show that $T$ spans a simplex of $L_\mathcal{U}$, it is enough to show that either $T \subset \Ge$ or the vertices of $T$ span a clique of $\Gamma$. 
	
	We can thus assume that $T$ is not contained in $\Ge$, and let us show that $T$ spans a clique of $\Gamma$. We decompose $T$ as a disjoint union $T = T_1 \cup T_2$, where $T_1 := T \cap V(\Ge) $ and $T_2 := T - T_1$. Note that by construction of $\mathcal{U}$, two vertices of $\Gamma$ are adjacent in $L_\mathcal{U}$ if and only if they are both contained in $\Ge$ or they are connected by an edge of $\Gamma$. By assumption, any two vertices of $T$ are adjacent in $L_{\mathcal{U}}$,  thus, for every $t \in T_2$ and $t' \in T - \{t\}$, we have that $t$ and $t'$ are connected by an edge in $\G$.
	
	It remains to show that any two vertices $t \neq t'\in T_1$ are connected by an edge of $\Gamma$. Suppose by contradiction that there exists a pair $t, t' \in T_1$ that is not connected by an edge of $\Gamma$. Since $T_2$ is not empty, we can pick an element $s \in T_2$ and the previous argument shows that $t, s, t'$ forms a path in $\Gamma$. Since $t, t'$ are not adjacent in $\Gamma$, this path is geodesic. By $2$-convexity of $\Ge$, we get that $s \in T_1$, a contradiction. Thus, the vertices of $T_1$ span a simplex of $\Gamma$, and it now follows that $T$ spans a clique of $\Gamma$, and hence spans a simplex in $L_\mathcal{U}$. 
\end{proof}

The following is now a direct consequence of Corollary~\ref{cor: IP with Godelle Paris}:

\begin{corollary}\label{cor: IP 2convex}
	Let $A_\Gamma$ be an Artin group, and let $\Ge$ be a $2$-convex subgraph of $\Gamma$. Assume that $A_{\Ge}$ as well as every clique standard parabolic subgroup of $A_\Gamma$ satisfy the Intersection Conjecture.	Then the intersection of any two conjugates of $A_{\Ge}$ is again a parabolic subgroup of $A_\Gamma$. 
\end{corollary}

\begin{proposition}\label{prop: AH splitting 2convex}
	Let $A_\Gamma$ be an irreducible Artin group that visually splits over a standard parabolic subgroup $A_{\Ge}$. Assume that: 
	\begin{itemize}
		\item $\Ge$ is $2$-convex in $\Gamma$,
		\item $A_\Ge$ and all clique parabolic subgroups of $A_\Gamma$ satisfy the Intersection Conjecture. 
	\end{itemize} 
	Then  $A_{\Gamma}$ is acylindrically hyperbolic. 
\end{proposition}

\begin{proof}
	It follows from Corollary~\ref{cor: IP 2convex} that the intersection of any two conjugates of $A_\Ge$ is again a parabolic subgroup of $A_\Gamma$. The result thus follows from Theorem~\ref{thm: ah splitting IP edge}.
\end{proof}

\paragraph{Application.}
We can use this result to obtain new examples of acylindrically hyperbolic Artin groups not covered by the recent results of Charney--Morris-Wright \cite{MorrisWright2019}.

\begin{definition}
	Let $C_n$ denotes the graph that is a cycle on $n$ vertices. The \textbf{wheel} $W_n$ is the graph obtained from $C_n$ by adding a new vertex (the apex) and connecting it to every vertex of $C_n$. 
\end{definition}

\begin{corollary}\label{cor:AH_wheel}
	Let $A_\Gamma$ be an irreducible Artin group whose underlying graph is a wheel $W_n$ with $n \geq 6$. Then  $A_{\Gamma}$ is acylindrically hyperbolic. 
\end{corollary}

\begin{proof}
	Since $n \geq 6$, $A_{W_n}$ visually splits over a parabolic subgroup $A_\Ge$ where  $\Ge$ is a geodesic of length $2$ containing the apex, that is $2$-convex in $W_n$. Since all $3$-generated Artin groups satisfy the Intersection Conjecture by \Cref{lem:IP for 3gen}, it follows that $A_\Ge$ and all clique parabolic subgroups of $A_\Gamma$ satisfy the Intersection Conjecture. The result now follows from Proposition~\ref{prop: AH splitting 2convex}.
\end{proof}

\section{The Weak Malnormality Conjecture} \label{sec:WM}

Next we consider the Weak Malnormality Conjecture.  
The following reduction lemma shows that it is enough to deal with irreducible Artin groups:

\begin{lemma}\label{lem: conjecture red from irred}
	Let $A_{\Gamma_1}, \ldots, A_{\Gamma_k}$ be irreducible Artin groups that satisfy the Weak Malnormality Conjecture. Then the direct product $A_{\Gamma_1}\times \cdots \times A_{\Gamma_k}$ also satisfies the  Weak Malnormality Conjecture. 
\end{lemma}

\begin{proof}
	Let $A_{\Gamma'}$ be a standard parabolic subgroup of the direct product $A_\Gamma := A_{\Gamma_1}\times \cdots \times A_{\Gamma_k}$. It is clear that if $A_{\Gamma'}$ contains one of the (normal) direct factors $A_{\Gamma_i}$, then it is not weakly malnormal. Thus, let us assume that $A_{\Gamma'}$ contains no direct factor. Note that $A_{\Gamma'}$ decomposes as the direct product $A_{\Gamma'} = A_{\Gamma_1'}\times \cdots \times A_{\Gamma_k'}$ where for each $i$,  $\Gamma_i' := \Gamma' \cap \Gamma_i$. Since $A_{\Gamma'}$ does not contain any of the $A_{\Gamma_i}$, each $A_{\Gamma_i'}$ is a proper parabolic subgroup of $A_{\Gamma_i}$, hence weakly malnormal in $A_{\Gamma_i}$ since by assumption, $A_{\Gamma_i}$ satisfies the Weak Malnormality Conjecture. Thus, for each $i$ we can pick $g_i \in A_{\Gamma_i}$ such that $A_{\Gamma_i'} \cap A_{\Gamma_i'}^{g_i}$ is finite. Now set $g := g_1\cdots g_k$. We have 
	$$A_{\Gamma'} \cap A_{\Gamma'}^g = \big( A_{\Gamma_1'} \cap A_{\Gamma_1'}^{g_1}\big) \times \cdots \times \big( A_{\Gamma_k'} \cap A_{\Gamma_k'}^{g_k} \big)  $$
	which is finite. Thus, $A_{\Gamma'}$ is weakly malnormal in $A_{\Gamma}$. 
\end{proof}

\subsection{Connections between the three main conjectures}

\begin{proposition}\label{prop: from wm edge group to every wm parabolic}
	Let $A_\Gamma$ be an irreducible Artin group that visually splits as an amalgamated product over a standard parabolic subgroup $A_{\Ge}$. If $A_{\Ge}$ is weakly malnormal in $A_{\Gamma}$, then every proper parabolic subgroup of $A_{\Gamma}$ is weakly malnormal.  Thus $A_\Gamma$ satisfies the Weak Malnormality Conjecture.
\end{proposition}

\begin{proof}
	Let $T$ be the Bass-Serre of the splitting $A_\Gamma = A_{\Gamma_1}*_{A_{\Ge}} A_{\Gamma_2}$. Let $P$ be a standard parabolic subgroup of $A_\Gamma$. Note that we have a splitting $P = P_1 *_{P_{\Ge}} P_2$, where $P_i := P \cap A_{\Gamma_i}$ and $P_{\Ge}= P \cap A_{\Ge}$. The Bass--Serre tree $T'$ of that induced splitting embeds isometrically in $T$. 
	
	Since $P$ is a proper parabolic subgroup, we have $T' \neq T$, hence we can pick a vertex $v$ of $T'$, and an edge $e=[v,w]$ that is not in $T'$.  Since $\mathrm{Stab}(e)$ is conjugated to $A_{\Ge}$, it is weakly malnormal in $A_\Gamma$. Choose $h \in A_\G$ such that $\mathrm{Stab}(e) \cap \mathrm{Stab}(he)$ is finite.  Let $\gamma$ be the geodesic path in $T$ with initial edge $e$ and final edge $he$. In particular, the pointwise stabilizer of $\gamma$ is finite.
	
		\medskip 
	
\textbf{Claim:} There exist elements $g_1, g_2 \in A_{\Gamma}$ such that the trees $g_1T'$ and $g_2T'$ are disjoint, and the unique geodesic between them contains $\gamma$. 

\medskip
	
 If either (or both) endpoints  of $\gamma$ are translates of $w$, we can extend $\gamma$ by a single edge at that endpoint to obtain a geodesic path $\gamma' \supseteq \gamma$ both of whose endpoints are translates of $v$.  Say the initial edge of $\gamma'$ is $g_1e$ and the final edge is $g_2 e$.  Since $e$ is not contained in $T'$, $g_ie$ is not contained in $g_iT'$ for $i=1,2$.  Thus $\gamma'$ intersects $g_iT'$ in a single point.   It now follows from standard arguments on the geometry of trees that $g_1T'$ and $g_2T'$ are disjoint, and $\gamma'$ is the unique geodesic between them, which proves the claim. 

\medskip
	
	Since $P$ stabilises the tree $T'$, the conjugate $P^{g_i}$ stabilises the tree $g_iT'$, and it follows that the intersection $P^{g_1} \cap P^{g_2}$ stabilises the unique geodesic between these disjoint trees.  Thus it fixes pointwise the path $\gamma'$. Since $\gamma' \supseteq \gamma$, it has finite stabiliser and we conclude that $P^{g_1} \cap P^{g_2}$ is finite.  Hence $P$ is weakly malnormal in $A_\Gamma$. 
\end{proof}

Combining the results above, we now conclude:

\begin{corollary}\label{cor: conj implications} Suppose $A_\G $ is irreducible and $\G$ is not a clique. 
	\begin{itemize}
		\item If $A_\G$ satisfies the Intersection Conjecture, then it also satisfies the Weak Malnormality Conjecture. 
		\item If $A_\G$ satisfies the Weak Malnormality Conjecture, then it also satisfies the Acylindrical Hyperbolicity Conjecture.  
	\end{itemize}
\end{corollary}

\begin{proof} If $\G$ is not a clique, $A_\G$ admits a visual splitting over some standard parabolic $A_\Ge$. The first bullet point is the direct application of Theorem~\ref{thm: ah splitting IP edge} and Proposition~\ref{prop: from wm edge group to every wm parabolic}.  For the second bullet point, since $A_\G$ is irreducible and satisfies the Weak Malnormality Conjecture, $A_{\Ge}$ is weakly malnormal, so the result follows from Theorem~\ref{thm: M-O}.
\end{proof}

One might wonder if the converse of these implications also hold.  
It is possible to obtain a partial converse to the implication in the second bullet point, assuming that~$A_\Gamma$ acts acylindrically on some hyperbolic space such that the geometry of the action is ``compatible'' with the parabolic subgroups, in the following sense: 

\begin{lemma}
	Let $A_\Gamma$ be an irreducible Artin group, and assume that $A_\Gamma$ is acylindrically hyperbolic, with a cobounded acylindrical action on a hyperbolic graph $X$. Suppose that for every proper parabolic subgroup $A_{\Gamma'}$, the following holds:  Let $X_{\Gamma'} \subset X$ denote the $A_{\Gamma'}$-orbit of some chosen point of $X$. Then $X_{\Gamma'}$ is quasi-convex in $X$ and its limit set $\Lambda X_{\Gamma'}$ is a strict subset of the Gromov boundary $\partial X$. 
	
	Then $A_\Gamma$ satisfies the Weak Malnormality Conjecture. 
\end{lemma}

\begin{proof}
	The set of limit points of loxdromic elements of $A_\Gamma$ is dense in $\partial X$ (see for instance Theorem 2.6 in \cite{Hamann2017}), so since $\Lambda X_{\Gamma'}$ is a proper closed subset of $\partial X$, we can pick a loxodromic element $g \in A_\Gamma$ such that $\Lambda g \cap \Lambda X_{\Gamma'} = \varnothing$. By hyperbolicity of $X$ and quasiconvexity of $X_{\Gamma'}$, there exist constants $\ell$ and $D$ (that depend only on the space $X$ and the quasiconvexity constants of $X_{\Gamma'}$) such that if two translates of $X_{\Gamma'}$ are at distance at least $\ell$, then the diameter of the closest projection of one on the other is bounded above by  $D$. By acylindricity of the action, we can pick a constant $L$ such that if $x, y\in X$ are at distance at least $L$, there are only finitely many elements $h \in A_\Gamma$ such that  $d(x, hx)\leq D$ and $ d(y, hx)\leq D$. 
	
	Using North-South dynamics of the action, we can now pick a large power $n \geq 0$ such that $X_{\Gamma'}$ and $g^nX_{\Gamma'}$ are disjoint, the diameter of the closest projection on each other is bounded above by $D$, and such that  their distance is greater than $L$. Let $x \in X_{\Gamma'}$ and $y\in g^nX_{\Gamma'}$ be a pair of points that realises the distance between these two translates. We get in particular that an element $h \in A_{\Gamma'} \cap A_{\Gamma'}^{g^n}$  sends the pair $x, y$ to another pair realising the distance between these two translates. Thus, for every $h \in A_{\Gamma'} \cap A_{\Gamma'}^{g^n}$, we have $d(x, hx)\leq D$ and $ d(y, hx)\leq D$. Since $d(x, y) \geq L$ by construction, the acylindricity implies that the set of such $h$ is finite. Thus, $A_{\Gamma'} \cap A_{\Gamma'}^{g^n}$ is finite, and $A_{\Gamma'}$ is weakly malnormal.
\end{proof}

Thus, we ask the following question: 

\begin{question} 
	Let $A_\Gamma$ be an irreducible Artin group, and assume that $A_\Gamma$ is acylindrically hyperbolic, with a cobounded acylindrical action on a hyperbolic graph $X$. Let $A_{\Gamma'}$ be a proper parabolic subgroup of $A_\Gamma$, and let $X_{\Gamma'} \subset X$ denote the $A_{\Gamma'}$-orbit of some chosen point of $X$. Do we have that $X_{\Gamma'}$ is quasi-convex in $X$, with limit set $\Lambda X_{\Gamma'} \neq \partial X$?
\end{question}

\subsection{Artin groups satisfying the Weak Malnormality Conjecture}
 
In this section we will show that the Weak Malnormality Conjecture holds for several classes of Artin groups, which allows us to prove that new classes of Artin groups are acylindrically hyperbolic.  

\begin{proposition}\label{prop:MC classes}
	The Weak Malnormality Conjecture holds for the following classes of groups: 
	\begin{itemize}
				\item Artin groups satisfying the hypothesis of Theorem~\ref{thm: ah splitting IP edge} (for instance, even Artin group of FC type),
		\item Artin groups of spherical type, 
		\item two-dimensional Artin groups.
\end{itemize}
\end{proposition}

Using the criterion of Minasyan--Osin \cite{minasyan_acylindrical_2015}, this implies the acylindrical hyperbolicity of many groups admitting a visual splitting:

\begin{corollary}\label{cor: AH for spherical even FC and Euclidean}
	Let $A_\Gamma$ be an Artin group with a visual splitting $A_{\G_1} *_{A_\Ge} A_{\G_2}$ and assume that $A_{\Ge}$ does not contain a direct factor of $A_{\Gamma_1}$ (which holds in particular if $A_{\Gamma_1}$ is irreducible).  Suppose that $A_{\Gamma_1}$ is one of the following: 
	\begin{itemize}
\item an Artin group satisfying the hypothesis of Theorem~\ref{thm: ah splitting IP edge} (for instance, an even Artin group of FC type),
\item an Artin group of spherical type, 
\item a two-dimensional Artin group.
	\end{itemize}
	Then $A_\Gamma$ is acylindrically hyperbolic. 
\end{corollary}

\begin{proof} By Proposition \ref{prop:MC classes}, the hypotheses of the corollary imply that $A_{\Ge}$ is weakly malnormal in $A_{\Gamma_1}$ and hence also in $A_\G$, so the result follows from Theorem \ref{thm: M-O}.
\end{proof}

The proof of Proposition \ref{prop:MC classes} will occupy the remainder of this section.

\paragraph{Spherical-type Artin groups.}

Although the Intersection Conjecture is known to hold for spherical-type Artin groups, we cannot apply Corollary~\ref{cor: conj implications} since the defining graph $\G$ is always a clique.  Nevertheless, we can prove:

\begin{lemma}\label{lem:spherical type}
	Artin groups of spherical type satisfy the Weak Malnormality Conjecture.
\end{lemma}

\begin{proof}
By Lemma~\ref{lem: conjecture red from irred},  it is enough to deal with the irreducible spherical case. Suppose that $A_\Gamma$ is irreducible and of spherical type, and let $A_{\Gamma'}$ be a proper parabolic subgroup. 
	
	If $A_\Gamma$ is not cyclic or of dihedral type, then by Theorem~3 of \cite{AntolinCumplido},  $A_\Gamma$ contains a subgroup isomorphic to $A_{\Gamma'}*\Z$, and where the free factor $A_{\Gamma'}$ is the proper parabolic subgroup under study. A standard argument from actions on trees shows that $A_{\Gamma'}$ is weakly malnormal in $A_{\Gamma'}*\Z$, hence it is weakly malnormal in $A_\Gamma$. 
	
	If $A_\Gamma$ is cyclic, there is nothing to prove. Suppose that $A_\Gamma$ is dihedral, with standard generators $s, t$, and let us show that $A_{\Gamma'} = \langle s \rangle$ is weakly malnormal.  We know from Lemma~\ref{Lem: no proper normal subgroups} that there exists $g \in A_\Gamma$ such that $\langle s \rangle^g \neq \langle s \rangle$. Let us show by contradiction that $\langle s \rangle^g \cap \langle s \rangle = \{1\}$, which will prove weak malnormality. Let $x \in \langle s \rangle^g \cap \langle s \rangle$ be a non-trivial element, and let $n, m \geq 1$ be such that $x = s^n = gs^mg^{-1}$. By applying the homomorphism $A_\Gamma \rightarrow \Z$ sending both generators to $1$, we see that $n=m$. Thus, $g$ lies in the centralizer $C(s^n)=C(s)$, the latter equality following for instance from Lemma 7 of \cite{Crisp2005}. It follows that $\langle s \rangle^g = \langle s \rangle$, a contradiction.
\end{proof}

\paragraph{Even Artin groups of FC-type.}~

	\begin{lemma}\label{lem:ThmA implies WM}
		Let $A_\G$ be an Artin group satisfying the hypotheses of Theorem~\ref{thm: ah splitting IP edge}. Then~$A_\G$ satisfies the Weak Malnormality Conjecture.
	\end{lemma}

\begin{proof}
	The edge group $A_\Ge$ is weakly malnormal in $A_\G$ by Theorem~\ref{thm: ah splitting IP edge}, so this is now a direct consequence of Proposition~\ref{prop: from wm edge group to every wm parabolic}.
\end{proof}

In particular, we get the following:

\begin{corollary}
	Even FC-type Artin groups satisfy the Weak Malnormality Conjecture. 
\end{corollary}

\begin{proof}
	By Lemma~\ref{lem: conjecture red from irred}, it is enough to assume that $A_\Gamma$ is irreducible and by Lemma  \ref{lem:spherical type} we may assume that it is not of spherical-type, that is, $\G$ is not a clique. The result now follows from Lemma~\ref{lem:ThmA implies WM}.
\end{proof}	

Note that the previous corollary is also a direct consequence of Corollary~\ref{cor: conj implications}.

\paragraph{Two-dimensional Artin Groups.}

An Artin group $A_\G$ is two-dimensional if $\G$ has at least one edge (i.e., $A_\G$ is not a free group) and any three vertices in $\G$ generate an infinite-type parabolic subgroup.  Recall that the Intersection Conjecture has not yet been proved for two-dimensional Artin groups, with currently the largest subclass for which it has been proved being the class of two-dimensional Artin groups whose presentation graph does not contain two adjacent edges with label $2$ \cite{Blufstein2022}. 

In this section we introduce another strategy for proving the Weak Malnormality Conjecture using an action of $A_\G$ on the Deligne complex, which we can apply to two-dimensional Artin groups.

\begin{proposition}\label{lem:malnormal-Deligne}
	Let $A_\Gamma$ be an Artin group such that:
	\begin{itemize}
		\item $D_\Gamma$ is CAT(0) with respect to either the cubical metric or the Moussong metric.
		\item There exists a vertex $v$ of $D_\Gamma$  with unbounded link and such that $\mbox{Stab}(v)$ is weakly malnormal in $A_\Gamma$.
	\end{itemize}
Then $A_\Gamma$ satisfies the Weak Malnormality Conjecture.
\end{proposition}

\begin{corollary} Two-dimensional Artin groups satisfy the Weak Malnormality Conjecture.
\end{corollary}

\begin{proof} First suppose $A_\Gamma$ contains an edge $e$ labelled $k >2$.  Then it cannot be reducible, since $e$ together with any vertex in the opposite direct factor would generate a spherical-type subgroup of rank 3. The Moussong metric on the Deligne complex $D_\Gamma$ is CAT(0), by Charney-Davis \cite{CharneyDavis1995}.  By Lemma 5.7 of Vaskou \cite{Vaskou2022}, there exists vertices $a,b$ in $\Gamma$ connected by an edge labelled $>2$ such that the subgroup $A_{a,b}$ is weakly malnormal in $A_\Gamma$. Viewing
	$A_{a,b}$ as a vertex in $D_\Gamma$, it has unbounded link by Proposition~E of~\cite{Vaskou2022}. Thus, we can apply Proposition \ref{lem:malnormal-Deligne} to conclude that every proper parabolic subgroup in $A_\Gamma$ is weakly malnormal. 
	
	If all edges of $\Gamma$ are labelled $2$ then $A_\Gamma$ is a RAAG, hence even FC-type, so the result follows from Theorem~\ref{thm: even FC-type}.
\end{proof}

\begin{proof}[Proof of Proposition~\ref{lem:malnormal-Deligne}]
For a proper parabolic subgroup $A_{\Gamma'}$, the Deligne  complex $D_{\Gamma'}$ embeds equivariantly as a strict convex subcomplex of the CAT(0) space $D_\Gamma$ that is stabilised by $A_{\Gamma'}$. (This is easily verified for the cubical metric.  For the Mosussong metric, see Lemma 5.1 of \cite{CharneyLocallyReducible}.)
We want to construct a translate $gD_{\Gamma'}$ such that the following is satisfied: 
\begin{itemize}
	\item There is a unique geodesic realising the distance between $D_{\Gamma'}$ and $gD_{\Gamma'}$.
	\item The pointwise stabiliser of that geodesic is finite.
\end{itemize}
This will imply that $A_{\Gamma'} \cap gA_{\Gamma'}g^{-1}$ is trivial, hence $A_{\Gamma'}$ is weakly malnormal.

Since checking that $A_{\Gamma'}$ is weakly malnormal is equivalent to checking that any of its conjugates is weakly malnormal, we will consider instead a translate $kD_{\Gamma'}$ for some $k \in A_{\Gamma}$, such that the vertex $v$ from the Proposition's statement is not contained in $kD_{\Gamma'}$.  We first observe that 
the projection of $kD_{\Gamma'}$ onto the link, $lk(v)$, has diameter at most $\pi$.  To see this, let $x,y$ be two points in  $kD_{\Gamma'}$ and let $\alpha_x, \alpha_y$ be the geodesics connecting $x$ to $v$ and $y$ to $v$.  If the angle between $\alpha_x$ and $\alpha_y$ was $\geq \pi$, a standard argument of CAT(0) geometry would imply that  the concatenation of $\alpha_x$ and $\alpha_y$ is  a geodesic from $x$ to $y$.  Since $kD_{\Gamma'}$ is convex in $D_\Gamma$, this geodesic must lie entirely in $kD_{\Gamma'}$.  This contradicts our assumption that $v \notin kD_{\Gamma'}$.

Since $\mbox{Stab}(v)$ is weakly malnormal, there exists a translate $w$ of $v$ such that $\mbox{Stab}(v) \cap \mbox{Stab}(w)$  is finite, and hence the geodesic $\gamma=[v,w]$  connecting them has finite pointwise stabiliser.  Since the link of $v$ is unbounded and $\mbox{Stab}(v)$ acts cocompactly on it, we can pick an element $h \in \mbox{Stab}(v)$ such that the distance in $lk(v)$ between the projection of $kD_{\Gamma'}$ and $h\gamma$ is at least $\pi$.  The stabiliser of $h\gamma$ is conjugate to that of $\gamma$, hence it is also finite.  Next, since the link of $hw$ is also unbounded, we can pick an element $g \in \mbox{Stab}(hw)$ such that the distance in $lk(hw)$ between $h\gamma$ and $gh\gamma$ is at least $\pi$.
And finally,  since the distance in $lk(v)$ between the projection of $kD_{\Gamma'}$ and $h\gamma$ is at least $\pi$ the same holds for the distance in $lk(gv)$  the between the projection of $gkD_{\Gamma'}$ and $gh\gamma$. (See Figure~1).

\medskip

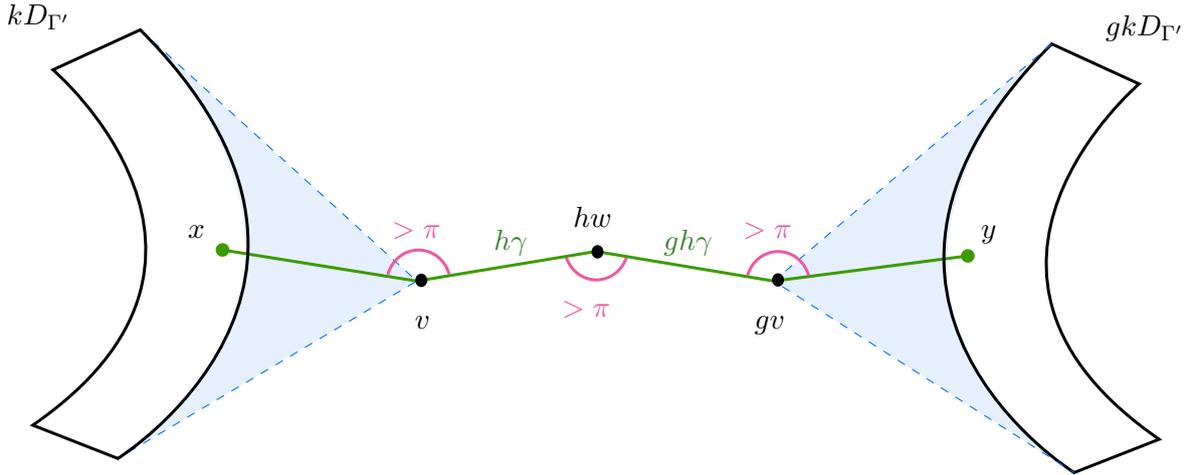
\begin{figure}[H]\label{figure2}
	\begin{center}
		\begin{overpic}[width=1\textwidth]{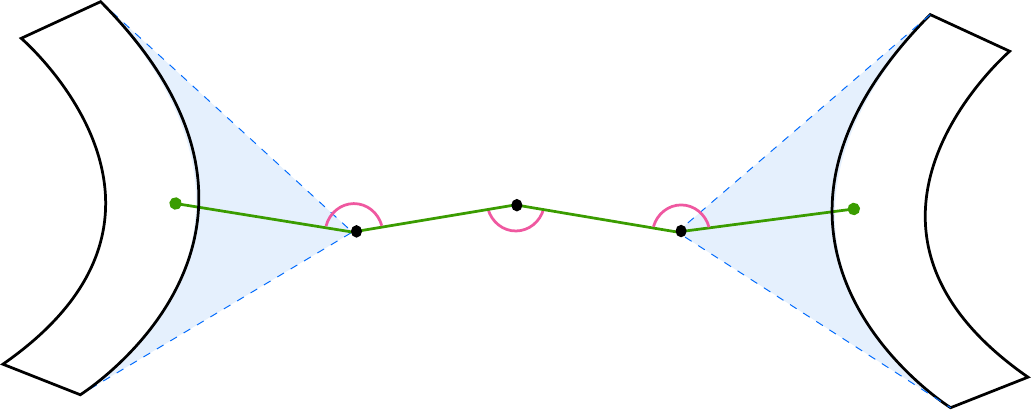}
			\put(14, 21){$x$}
			\put(34, 13){$v$}
			\put(32, 21){{\color{VioletRed}$> \pi$}}
			\put(48, 22){$hw$}
			\put(47, 14){{\color{VioletRed}$> \pi$}}
			\put(64, 13){$gv$}
			\put(63, 21){{\color{VioletRed}$> \pi$}}
			\put(84, 21){$y$}
				\put(41, 20){{\color{OliveGreen}$h\gamma$}}
				\put(56, 20){{\color{OliveGreen}$gh\gamma$}}
			\put(-2, 40){$k D_{\Gamma'}$}
			\put(95, 39){$gk D_{\Gamma'}$}
		\end{overpic}
	\end{center}
	\caption{A geodesic (green) between $kD_{\Gamma'}$ and $gkD_{\Gamma'}$, obtained by concatenating several geodesic segments making an angle greater than $\pi$ at their intersection point. The angles at $v$ between any two points of $kD_{\Gamma'}$ are  smaller than $\pi$ (``visual cone" in light blue).}
\end{figure}

It follows that for any points $x \in kD_{\Gamma'}$ and $y \in gkD_{\Gamma'}$, the concatenation of the geodesics $[x,v], [v,hw], [hw,gv], [gv,y]$ is the (unique) geodesic from $x$ to $y$.  In particular, taking $x$ to be the nearest point projection of $v$ on $kD_{\Gamma'}$ and $y$ to be the nearest point projection of $gv$ on $gkD_{\Gamma'}$, we obtain a unique length-minimizing path between $kD_{\Gamma'}$ and~$gkD_{\Gamma'}$.  Since $kA_{\Gamma'}k^{-1} \cap gkA_{\Gamma'}(gk)^{-1}$ preserves both of these subcomplexes, it must fix this path.  In particular, it lies in the pointwise stabilser of $h\gamma$.  We conclude that this intersection is finite and hence $A_{\Gamma'}$  is weakly malnormal.
\end{proof}

\bibliographystyle{alpha}
\bibliography{acy-bibliography}

\begin{thebibliography}{CGGMW19}

\bibitem[AC21]{AntolinCumplido}
Yago Antol{\'i}n and Mar{\'i}a Cumplido.
\newblock Parabolic subgroups acting on the additional length graph.
\newblock {\em Algebraic \& Geometric Topology}, 2021:1791–1816, 2021.

\bibitem[AF22]{antolin_subgroups_2023}
Yago Antol{\'i}n and Islam Foniqi.
\newblock Intersection of parabolic subgroups in even {Artin} groups of {FC}
  type.
\newblock {\em Proceedings of the Edinburgh Mathematical Society}, 65:938--957,
  11 2022.

\bibitem[AM15]{Antolin2015}
Yago Antolín and Ashot Minasyan.
\newblock {Tits} alternatives for graph products.
\newblock {\em Journal fur die Reine und Angewandte Mathematik}, 704:55--83,
  2015.

\bibitem[Blu22]{Blufstein2022}
Martin Blufstein.
\newblock Parabolic subgroups of two-dimensional {Artin} groups and
  systolic-by-function complexes.
\newblock {\em Bulletin of the London Mathematical Society}, 54, 2022.

\bibitem[Bow08]{Bowditch2008}
Brian Bowditch.
\newblock Tight geodesics in the curve complex.
\newblock {\em Inventiones Mathematicae}, 171:281--300, 2008.

\bibitem[BP23]{blufstein_parabolic_2023}
Martin Blufstein and Luis Paris.
\newblock Parabolic subgroups inside parabolic subgroups of {Artin} groups.
\newblock {\em Proceedings of the American Mathematical Society},
  151:1519--1526, 2023.

\bibitem[Cal22]{Calvez2022}
Matthieu Calvez.
\newblock Euclidean {Artin-Tits} groups are acylindrically hyperbolic.
\newblock {\em Groups,Geometry, and Dynamics}, 16:963--983, 2022.

\bibitem[CD95]{CharneyDavis1995}
Ruth Charney and Michael Davis.
\newblock The {K}($\pi$,1)-problem for hyperplane complements associated to
  infinite reflection groups.
\newblock {\em Journal of the American Mathematical Society}, 8(3):597--627,
  1995.

\bibitem[CGGMW19]{CumplidoGebhardtGonzalesMenesesWiest2017}
María Cumplido, Volker Gebhardt, Juan González-Meneses, and Bert Wiest.
\newblock On parabolic subgroups of {Artin}-{Tits} groups of spherical type.
\newblock {\em Advances in Mathematics}, 352:572--610, 2019.

\bibitem[Cha00]{CharneyLocallyReducible}
Ruth Charney.
\newblock The {Tits} conjecture for locally reducible {Artin} groups.
\newblock {\em International Journal of Algebra and Computation},
  10(6):783--797, 2000.

\bibitem[Cha04]{Charney2004}
Ruth Charney.
\newblock The {Deligne} complex for the four-strand braid group.
\newblock {\em Transactions of the American Mathematical Society},
  356(10):3881--3897, 2004.

\bibitem[CM16]{ChatterjiMartin2019}
Indira Chatterji and Alexandre Martin.
\newblock {\em A note on the acylindrical hyperbolicity of groups acting on
  {CAT}(0) cube complexes}, pages 160--178.
\newblock London Mathematical Society Lecture Notes Series. Cambridge
  University Press, August 2016.

\bibitem[CMV23]{CumplidoMartinVaskou2020}
María Cumplido, Alexandre Martin, and Nicolas Vaskou.
\newblock Parabolic subgroups of large-type {Artin} groups.
\newblock {\em Mathematical Proceedings of the Cambridge Philosophical
  Society}, 174:393--414, 2023.

\bibitem[CMW19]{MorrisWright2019}
Ruth Charney and Rose Morris-Wright.
\newblock Artin groups of infinite type: trivial centers and acylindical
  hyperbolicity.
\newblock {\em Proceedings of the American Mathematical Society},
  147(9):3675--3689, 2019.

\bibitem[Cri05]{Crisp2005}
John Crisp.
\newblock Automorphisms and abstract commensurators of two-dimensional {Artin}
  groups.
\newblock {\em Geometry and Topology}, 9:1381--1441, 2005.

\bibitem[CW17]{CalvezWiest2017Acylindrical}
Matthieu Calvez and Bert Wiest.
\newblock Acylindrical hyperbolicity and {Artin}-{Tits} groups of spherical
  type.
\newblock {\em Geometriae Dedicata}, 191(1):199--215, 2017.

\bibitem[Dro87]{Droms1987}
Isomorphisms of graph groups.
\newblock {\em Proceedings of the American Mathemactical Society},
  100(3):407--408, 1987.

\bibitem[Gol22]{Goldman2022}
Katherine Goldman.
\newblock The {K$(\pi,1)$} conjecture and acylindrical hyperbolicity for
  relatively extra-large {Artin} groups.
\newblock {\em Algebraic and Geometric Topology}, 2022.
\newblock to appear. available at \textsf{https://arxiv.org/abs/2211.16391}.

\bibitem[GP12a]{GodelleParis2012Basic}
Eddy Godelle and Luis Paris.
\newblock Basic questions on {Artin}-{Tits} groups.
\newblock In {\em Configuration Spaces}, pages 299--311. 2012.

\bibitem[GP12b]{GodelleParis2012KPi1}
Eddy Godelle and Luis Paris.
\newblock K($\pi$, 1) and word problems for infinite type {Artin}-{Tits}
  groups, and applications to virtual braid groups.
\newblock {\em Mathematische Zeitschrift}, 272(3-4):1339--1364, 2012.

\bibitem[Hae22]{HaettelXXL}
Thomas Haettel.
\newblock {XXL} type {Artin} groups are {CAT}(0) and acylindrically hyperbolic.
\newblock {\em Annales de l'Institut Fourier}, 72(6):2541--2555, 2022.

\bibitem[Hae24]{haettel_lattices_2023}
Thomas Haettel.
\newblock Lattices, injective metrics and the {K$(\pi,1)$} conjecture.
\newblock {\em Algebraic \& Geometric Topology}, 2024.
\newblock to appear. available at \textsf{https://arxiv.org/abs/2109.07891v5}.

\bibitem[Ham17]{Hamann2017}
Matthias Hamann.
\newblock Group actions on metric spaces: fixed points and free subgroups.
\newblock {\em Abhandlungen aus dem Mathematischne Seminar der Universit{\"a}t
  Hamburg}, 87:245--263, 2017.

\bibitem[HMS24]{HagenMartinSistoHHG}
Mark Hagen, Alexandre Martin, and Alessandro Sisto.
\newblock Extra-large type {Artin} groups are hierarchically hyperbolic.
\newblock {\em Mathematische Annalen}, 388:867--938, 2024.

\bibitem[HO19]{HuangOsajdaSystolic}
Jingyin Huang and Damian Osajda.
\newblock Metric systolicity and two-dimensional {Artin} groups.
\newblock {\em Mathematische Annalen}, 374:1311--1352, 2019.

\bibitem[HO21]{HuangOsajdaHelly}
Jingyin Huang and Damian Osajda.
\newblock Helly meets {Garside} and {Artin}.
\newblock {\em Inventiones Mathematicae}, 225:395--426, 2021.

\bibitem[KO22]{KatoOguni2022}
Motoko Kato and Shin-Ichi Oguni.
\newblock Acylindrical hyperbolicity of {Artin-Tits} groups associated with
  triangle-free graphs and cones over square-free bipartite graphs.
\newblock {\em Glasgow Mathematical Journal}, 64:51--64, 2022.

\bibitem[Mas24]{JillThesis}
Jill Mastrocola.
\newblock {\em Properties of locally reducible {Artin} groups}.
\newblock PhD thesis, 2024.
\newblock PhD thesis: Brandeis University.

\bibitem[MM99]{MasurMinsky1999}
Howard Masur and Yair Minsky.
\newblock Geometry of the complex of curves {I}: hyperbolicity.
\newblock {\em Inventiones Mathematicae}, 138:103--149, 1999.

\bibitem[MO15]{minasyan_acylindrical_2015}
Ashot Minasyan and Denis Osin.
\newblock Acylindrical hyperbolicity of groups acting on trees.
\newblock {\em Mathematische Annalen}, 362(3-4):1055--1105, 2015.

\bibitem[MP22]{martin_acylindrical_2021}
Alexandre Martin and Piotr Przytycki.
\newblock Acylindrical actions for two-dimensional {Artin} groups of hyperbolic
  type.
\newblock {\em International Mathematics Research Notes},
  2022(17):13099–13127, 2022.

\bibitem[MPV23]{moller_parabolic_2023}
Philip M\"oller, Luis Paris, and Olga Varghese.
\newblock On parabolic subgroups of {Artin} groups.
\newblock {\em Israel Journal of Mathematics}, 2023.

\bibitem[MV24]{MartinVaskou2024}
Alexandre Martin and Nicolas Vaskou.
\newblock Characterising large-type {Artin} groups.
\newblock {\em Bulletin of the London Mathematical Society}, 2024.
\newblock \textit{to appear}, available at
  \textsf{https://arxiv.org/abs/2305.05702}.

\bibitem[MW21]{MorrisWright2021}
Rose Morris-Wright.
\newblock Parabolic subgroups in {FC} type {Artin} groups.
\newblock {\em Journal of Pure and Applied Algebra}, 225(1), 2021.

\bibitem[Osi16]{Osin2016}
Denis Osin.
\newblock Acylindrically hyperbolic groups.
\newblock {\em Transactions of the American Mathematical Society},
  368(2):851--888, 2016.

\bibitem[Osi17]{Osin2017}
Denis Osin.
\newblock Groups acting acylindrically on hyperbolic spaces.
\newblock In {\em Proceedings of the International Congress of Mathematicians
  (ICM 2018)}, pages 919--939. 2017.

\bibitem[Par04]{Paris2004}
Luis Paris.
\newblock Artin groups of spherical type up to isomorphism.
\newblock {\em Journal of Algebra}, 281:666--678, 2004.

\bibitem[Vas22]{Vaskou2022}
Nicolas Vaskou.
\newblock Acylindrical hyperbolicity for {Artin} groups of dimension 2.
\newblock {\em Geometriae Dedicata}, 216(7), 2022.

\bibitem[Vas23]{Vaskou2023isomorphism}
Nicolas Vaskou.
\newblock The isomorphism problem for large-type {Artin} groups, 2023.
\newblock preprint available at \textsf{https://arxiv.org/abs/2201.08329v3}.

\bibitem[VdL83]{VanderLek1983}
Harm Van~der Lek.
\newblock {\em The homotopy type of complex hyperplane complements}.
\newblock PhD thesis, 1983.
\newblock Publication Title: PhD thesis, Nijmegen.

\end{thebibliography}

\end{document}